\documentclass{amsart}
\usepackage{enumerate}
\newtheorem{theorem}{Theorem}[section]

\usepackage{epsfig}
\usepackage{epstopdf}
\newtheorem{definition}[theorem]{Definition}

\newtheorem{proposition}{Proposition}[section]
\newtheorem{assumption}{Assumption}[section]

\theoremstyle{remark}
\newtheorem{corollary}{Corollary}[section]
\newtheorem{remark}[theorem]{Remark}
\usepackage{graphicx}
\numberwithin{equation}{section}

\begin{document}

\title[Inverse boundary value problem for time-harmonic elastic
  waves]{Uniqueness and Lipschitz stability of an inverse
  boundary value problem for time-harmonic elastic waves}

\author{E. Beretta}
\address{Dipartimento di Matematica ``F. Brioschi'', Politecnico di Milano}
\curraddr{}
\email{elena.beretta@polimi.it}
\thanks{}

\author{M.V. de Hoop}
\address{Department of Mathematics, Purdue University}
\curraddr{}
\email{mdehoop@purdue.edu}
\thanks{}

\author{E. Francini}
\address{Dipartimento di Matematica e Informatica ``U. Dini'', Universit\`{a} di Firenze}
\curraddr{}
\email{elisa.francini@unifi.it}
\thanks{}

\author{S. Vessella}
\address{Dipartimento di Matematica e Informatica ``U. Dini'', Universit\`{a} di Firenze}
\curraddr{}
\email{sergio.vessella@unifi.it}
\thanks{}

\author{J. Zhai}
\address{Department of Mathematics, Purdue University}
\curraddr{}
\email{zhai12@purdue.edu}
\thanks{}



\date{}

\dedicatory{}

\begin{abstract}
We consider the inverse problem of determining the Lam\'{e} parameters
and the density of a three-dimensional elastic body from the local
time-harmonic Dirichlet-to-Neumann map. We prove uniqueness and
Lipschitz stability of this inverse problem when the Lam\'{e}
parameters and the density are assumed to be piecewise constant on a
given domain partition.
\end{abstract}

\maketitle

\section{Introduction}

We study the inverse boundary value problem for time-harmonic elastic
waves. We consider isotropic elasticity, and allow partial boundary
data. The Lam\'{e} parameters and the density are assumed to be
piecewise constants on a given partitioning of the domain. The system
of equations describing time-harmonic elastic waves is given by,
\begin{equation}\label{elasticity system}
\begin{cases}
\operatorname{div}(\mathbb{C}\hat{\nabla} u)+\rho\omega^2 u=0&\text{in}~\Omega\subset\mathbb{R}^3 ,
\\
u=\psi&\text{on}~~\partial\Omega ,
\end{cases}
\end{equation}
where $\Omega$ is an open and bounded domain with smooth boundary,
$\hat{\nabla}u$ denotes the strain tensor,
$\hat{\nabla}u:=\frac{1}{2}(\nabla u+(\nabla u)^T)$, $\psi\in
H^{1/2}(\partial\Omega)$ is the boundary displacement or source,
and $\mathbb{C}\in L^\infty(\Omega)$ denotes the isotropic elasticity
tensor with Lam\'{e} parameters $\lambda,\mu$:
\[\mathbb{C}=\lambda I_3\otimes I_3+2\mu\mathbb{I}_{sym},~\text{a.e. in}~\Omega,\]
where $I_3$ is $3\times 3$ identity matrix and $\mathbb{I}_{sym}$ is the fourth order tensor such that $\mathbb{I}_{sym}A=\hat{A}$, $\rho\in L^\infty(\Omega)$ is the density, and $\omega$ is the frequency. Here, we make use of the following notation for matrices and tensors: For $3\times 3$ matrices $A$ and $B$ we set $A:B=\sum_{i,j=1}^3A_{ij}B_{ij}$ and $\hat{A}=\frac{1}{2}(A+A^T)$. We assume that
\begin{eqnarray}
&0<\alpha_0\leq\mu\leq \alpha_0^{-1}, 0<\beta_0\leq 2\mu+3\lambda\leq\beta_0^{-1}~\text{a.e. in}~\Omega,\label{1.2}\\
&0\leq\rho\leq\gamma_0^{-1}.\label{1.3}
\end{eqnarray}
The Dirichlet-to-Neumann map, $\Lambda_{\mathbb{C},\rho}$, is defined
by
\[
   \Lambda_{\mathbb{C},\rho} : H^{1/2}(\partial\Omega)
      \ni \psi \rightarrow (\mathbb{C} \hat{\nabla}u)
        \nu|_{\partial\Omega} \in H^{-1/2}(\partial\Omega) ,
\]
where $\nu$ is the outward unit normal to $\partial\Omega$. We
consider the inverse problem:
\[
   \text{determine
           $\mathbb{C},\rho$ from $\Lambda_{\mathbb{C},\rho}$.}
\]
For the static case (that is, $\omega=0$) of our problem, Imanuvilov
and Yamamoto \cite{IY} proved, in dimension two, a uniqueness result
for $C^{10}$ Lam\'{e} parameters. In dimension three, Nakamura and
Uhlmann \cite{NU2} proved uniqueness assuming that the Lam\'{e}
parameters are $C^{\infty}$ and that $\mu$ is close to a positive
constant. Eskin and Ralston \cite{ER} proved a related result. Global
uniqueness of the inverse problem in dimension three assuming general
Lam\'{e} parametres remains an open problem. Beretta \textit{et al.}
proved the uniqueness when the Lam\'{e} parameters are assumed to be
piecewise constant. They proved the Lipschitz stability when
interfaces of subdomains contain flat parts \cite{BFV}; later, they
extended this result to non-flat interfaces \cite{BFMRV}. Alessandrini
\textit{et al.} \cite{AdCMR} proved a logarithmic stabilty estimate
for the inverse problem of identifying an inclusion, where constant
Lam\'{e} parameters are different from the background ones.

The key application we have in mind is (reflection) seismology, where
Lam\'{e} parameters and density need to be recovered from the
Dirichlet-to-Neumann map. In actual seismic acquisition, raw vibroseis
data are modeled by the Neumann-to-Dirichlet map, the inverse of the
Dirichlet-to-Neumann map: The boundary values are given by the normal
traction underneath the base plate of a vibroseis and are zero (`free
surface') elsewhere, while the particle displacement (in fact,
velocity) is measured by geophones located in a subset of the boundary
(Earth's surface). The applied signal is essentially time-harmonic
(suppressing the sweep); see \cite[(2.52)-(2.53)]{Bae}. (The
displacement needs to be measured also underneath the base plate.)

A key complication addressed in this paper is the multiparameter
aspect of this inverse problem. For the acoustic waves modeled by the equation
\begin{equation}\label{acoustic}\nabla\cdot(\gamma\nabla u)+q\omega^2 u=0,\end{equation}
Nachman \cite{Nach} proved the unique recovery of $\gamma\in C^2$ and $q\in L^\infty$ with Dirichlet-to-Neumann maps at two different admissible frequencies $\omega_1,\omega_2$. For the optical tomography problem, that is, recovering simultaneously $a>0$ and $c>0$ in the partial differential equation
\[-\nabla\cdot(a\nabla u)+cu=0,\]
from all possible boundary Dirichlet and Neumann pairs, Arridge and Lionheart \cite{AL98} demonstrated the non-uniqueness for general $a$ and $c$. However, when $a$ is piecewise constant and $c$ is piecewise analytic, Harrach \cite{Ha} proved the uniqueness of this inverse problem. In this paper, we prove, for our problem, that recovering a higher order coefficient and a lower order coefficient jointly, that are assumed to be piecewise constant, only needs single frequency data also. If we assume $\gamma,q$ to be piecewise constant in (\ref{acoustic}), we can establish the uniqueness with single frequency data, following the methods of proof of this paper.

With the conditional Lipschitz
stability which we obtain here, we can invoke iterative methods with
guaranteed convergence for local reconstruction, such as the nonlinear
Landweber iteration \cite{dHQS1} and the nonlinear projected
steepest descent algorithm \cite{dHQS2} (including a stopping
criterion which allows inaccurate data). In reflection seismology,
iterative methods for solving inverse problems, casting these into
optimization problems, have been collectively referred to as Full
Waveform Inversion (FWI) through the use of the adjoint state
method. These methods were introduced in this field of application by
Chavent \cite{Chavent:1983}, Lailly \cite{Lailly:1983} and Tarantola
\& Valette \cite{Tarantola:1982, Tarantola:1984} albeit for scalar
waves. An early study of stability in dimension one can be found in
Bamberger \textit{et al.} \cite{Bamberger:1979}. Mora \cite{Mora:1987}
developed the adjoint state formulation for the case of elastic waves
and carried out computational experiments; Crase \textit{et al.}
\cite{Crase:1990} then carried out applications to field
data. Advantages of using time-harmonic data, following specific
workflows, were initially pointed out by Pratt and collaborators
\cite{PrattWorthington:1990, Pratt:1996, Pratt:1999}; Bunks \textit{et
  al.} \cite{Bunks:1995} developed an important insight in the use of
strictly finite-frequency data. In recent years, there has been a
significant effort in further developing and applying these approaches
(with emphasis on iterative Gauss-Newton methods) -- in the absence of
a notion of (conditional) uniqueness, stability or convergence --
often in combination with intuitive strategies for selecting parts of
the data. In exploration seismology, we mention the work of G\'{e}lis
\textit{et al.} \cite{Gelis:2007}, Choi \cite{Choi:2008}, Brossier
\textit{et al.} \cite{Brossier:2009, Brossier:2010} and Xu \& McMechan
\cite{McMechan:2014}; in global seismology, we mention the work of
Tromp \textit{et al.} \cite{Tromp:2005} and Fichtner \& Trampert
\cite{Fichtner:2011}.

In this paper, we consider piecewise constant Lam\'{e} parameters and
density of the form
\[
   \mathbb{C}(x) = \sum_{j=1}^N (\lambda_j I_3 \otimes I_3
       + 2\mu_j\mathbb{I}_{sym})\chi_{D_j}(x) ,\quad
   \rho(x) = \sum_{j=1}^N\rho_j\chi_{D_j}(x) ,
\]
where the $D_j$'s, $j=1,\cdots,N$ are known disjoint Lipschitz domains
and $\lambda_j,\mu_j,\rho_j,j=1,\cdots,N$ are unknown constants. We
establish uniqueness and a Lipschitz stability estimate of the above
mentioned inverse boundary value problem. The method of proof follows
the ideas introduced by Alessandrini and Vessella \cite{Al2} in the
study of electrical impedance tomography (EIT) problems. The
counterpart for scalar waves, that is, the inverse boundary value
problem for the Helmholtz equation, was analyzed by Beretta \textit{et
  al.} \cite{BdHQ}.

The existence and the ``blow up'' behavior of singular solutions close
to a flat discontinuity are utilized in our proof. The quantitative
estimate of unique continuation for elliptic systems, which is derived
from a three spheres inequality, play an essential role in the
procedure. We directly prove a log-type stability estimate for the
Lam\'{e} parameters and the density combined with alternatingly
estimating them along a walkway of subdomains. Uniqueness then follows
from the stability estimate. From the restriction that the parameters to be
recovered lie in a finite-dimensional space, a Lipschitz stability
estimate is obtained.

The paper is organized as follows: In Section 2, we summarize the main
results. In Section 3, we construct the singular solutions and
establish the unique continuation for the system describing
time-harmonic elastic waves. We also prove the Fr\'{e}chet
differentiability of the forward map, $(\mathbb{C},\rho)
\to \Lambda_{\mathbb{C},\rho}$. In Section 4, we prove the main
result. In Section 5, we give some remarks on the problems of
identifying the Lam\'{e} parameters given the density, and identifying
the density given the Lam\'{e} parameters.

\section{Main result}

\subsection{Direct problem}

We summarize some results concerning the well-posedness of problem
(\ref{elasticity system}).

\begin{proposition}\label{direct problem}
Let $\Omega$ be a bounded Lipschitz domain in $\mathbb{R}^3$, $f\in
H^{-1}(\Omega)$ and $g\in H^{1/2}(\partial\Omega)$. Assume that
$\lambda,\mu,\rho$ satisfy (\ref{1.2}) and (\ref{1.3}). Let
$\lambda_1^0$ be the smallest Dirichlet eigenvalue of the operator
$-\operatorname{div}(\mathbb{C}_0\hat{\nabla}u)$ in $\Omega$, where
$\mathbb{C}_0=\frac{\beta_0-3\alpha_0}{2} I_3\otimes
I_3+2\alpha_0\mathbb{I}_{sym}$. Then, for any $\omega^2\in
(0,\frac{\gamma_0 \lambda_1^0}{2}]$, there exists a unique solution of
\begin{equation}\label{2.3}
\begin{cases}
\operatorname{div}(\mathbb{C}\hat{\nabla} u)+\rho\omega^2 u=f&\text{in}~\Omega\subset\mathbb{R}^3 ,\\
u=g&\text{on}~~\partial\Omega ,
\end{cases}
\end{equation}
 satisfying
\begin{equation}\label{2.2}\|u\|_{H^1(\Omega)}\leq C(\|g\|_{H^{1/2}(\partial\Omega)}+\|f\|_{H^{-1}(\Omega)}),
\end{equation}
 where $C$ depends on $\alpha_0$, $\beta_0$, $\gamma_0$ and $\lambda_1^0$.
\end{proposition}

\begin{proof} 
Without loss of generality, we let $g = 0$. Indeed, we can always
introduce a $w = u - \tilde{g}$ where $\tilde{g} \in H^1(\Omega)$ is
such that $\tilde{g} = g$ on $\partial\Omega$, which satisfies
(\ref{2.3}) with $g = 0$. We recall that
\begin{equation}
\label{*} 
   \lambda_1^0 = \min \left\{ \int_{\Omega}
        \mathbb{C}_0 \hat{\nabla}u : \hat{\nabla}u \,
   \Big| \,u\in H^1(\Omega),\,\|u\|_{L^2(\Omega)}=1\right\} ,
\end{equation}
and observe that $\mathbb{C}\geq\mathbb{C}_0$, that is,
$(\mathbb{C}-\mathbb{C}_0)\hat{A}:\hat{A} \geq 0$ for any $3\times 3$
matrix $A$.

We consider on $H^1_0(\Omega)$ the bilinear form 
\[a(u,v)=\int_{\Omega}\mathbb{C}\hat{\nabla}u:\hat{\nabla}v\mathrm{d}x - \int_{\Omega}\omega^2\rho u\cdot v\mathrm{d}x.\]
Then we can write problem (\ref{2.3}) (for $g=0$) in the weak form,
\[a(u,v)=- \langle f,v \rangle\quad\forall v\in H^1_0(\Omega).\]
Clearly $a(\cdot,\cdot)$ is continuous. We check now that $a(\cdot,\cdot)$ is coercive. To this aim, we recall the Korn inequality
\begin{equation}
\label{**}
   \int_{\Omega} |\hat{\nabla} u|^2\mathrm{d}x
         \leq 2 \int_{\Omega} |\nabla u|^2\mathrm{d}x
\end{equation}
for any $u\in H^1_0(\Omega)$ (using the matrix norm,
$\left|A\right|^2=A:A$ for any $3\times 3$ matrix $A$). Furthermore,
\begin{eqnarray*}
	a(u,u)&=&\int_{\Omega} \mathbb{C}\hat{\nabla}u:\hat{\nabla}u\mathrm{d}x-\int_{\Omega}\omega^2\rho |u|^2\mathrm{d}x\\
	&\geq& \int_{\Omega}\mathbb{C}_0\hat{\nabla}u:\hat{\nabla}u\mathrm{d}x-\omega^2\gamma_0^{-1}\int_{\Omega} |u|^2\mathrm{d}x\\
   &=&\frac{1}{2}\int_{\Omega} \mathbb{C}_0\hat{\nabla}u:\hat{\nabla}u\mathrm{d}x
   + \frac{1}{2} 
     \left\{ \int_{\Omega}\mathbb{C}_0\hat{\nabla}u:\hat{\nabla}u\mathrm{d}x-2\omega^2\gamma_0^{-1}\int_{\Omega} |u|^2\mathrm{d}x\right\} .
\end{eqnarray*}
By (\ref{*}), the strong convexity of $\mathbb{C}_0$, the Korn inequality (\ref{**}) and the Poincar\'{e} inequality, we have
\begin{eqnarray*}
	a(u,u)	&\geq& \frac{\xi_0}{4}\int_{\Omega}\left|\nabla u\right|^2\mathrm{d}x+\frac{1}{2}\left\{ \int_{\Omega}\mathbb{C}_0\hat{\nabla}u:\hat{\nabla}u\mathrm{d}x-2\omega^2\gamma_0^{-1}\int_{\Omega} |u|^2\mathrm{d}x\right\}\\
	&\geq&\frac{\xi_0C_P}{4}\|u\|_{H^1(\Omega)}^2
\end{eqnarray*}
indeed, where $\xi_0$ depends on $\alpha_0$ and $\beta_0$ only and
$C_P$ is the Poincar\'{e} constant of $\Omega$. By the Lax-Milgram
lemma there exists a unique solution $u\in H^1_0(\Omega)$ to problem
(\ref{2.3}), and (\ref{2.2}) holds.
\end{proof}

\begin{remark}
We note that whenever $\omega$ is not in a particular countable subset
of real numbers (the set of eigenfrequencies), Problem (\ref{direct problem}) has a unique solution and estimate (\ref{2.2}) holds with the constant $C$ depending also on $\omega$.
\end{remark}

We let $\Sigma$ be an open portion of $\partial\Omega$. We denote by
$H^{1/2}_{co}(\Sigma)$ the space
\[
   H^{1/2}_{co}(\Sigma) := \{ \phi \in H^{1/2}(\partial\Omega)
      \ |\ \operatorname{supp}~ \phi \subset\Sigma \}
\]
and by $H^{-1/2}_{co}(\Sigma)$ the topological dual of
$H^{1/2}_{co}(\Sigma)$. We denote by $\langle\cdot,\cdot\rangle$ the
dual pairing between $H^{1/2}_{co}(\Sigma)$ and
$H^{-1/2}_{co}(\Sigma)$ based on the $L^2(\Sigma)$ inner product. By
Proposition \ref{direct problem} it follows that for any $\psi \in
H^{1/2}_{co}(\Sigma)$ there exists a unique vector-valued function $u
\in H^1(\Omega)$ that is a weak solution of the Dirichlet problem
(\ref{elasticity system}). We define the local Dirichlet-to-Neumann
map $\Lambda_{\mathbb{C},\rho}^{\Sigma}$ as
\[
   \Lambda_{\mathbb{C},\rho}^{\Sigma} : H^{1/2}_{co}(\Sigma)
        \ni \psi \rightarrow (\mathbb{C} \hat{\nabla} u)
                     \nu|_{\Sigma} \in H^{-1/2}_{co}(\Sigma) .
\]
We have $\Lambda_{\mathbb{C},\rho} =
\Lambda_{\mathbb{C},\rho}^{\partial\Omega}$. The map
$\Lambda_{\mathbb{C},\rho}^{\Sigma}$ can be identified with the
bilinear form on $H^{1/2}_{co}(\Sigma) \times H^{-1/2}_{co}(\Sigma)$,
\begin{equation}\label{eq:bilform}
\hat{\Lambda}_{\mathbb{C},\rho}^\Sigma(\psi,\phi):=\langle\Lambda_{\mathbb{C},\rho}^\Sigma\psi,\phi\rangle=\int_{\Omega}(\mathbb{C}\hat{\nabla}u:\hat{\nabla}v-\rho\omega^2 u \cdot v)\mathrm{d}x ,
\end{equation}
for all $\psi,\phi \in H^{1/2}_{co}(\Sigma)$, where $u$ solves
(\ref{elasticity system}) and $v$ is any $H^1(\Omega)$ function such
that $v=\phi$ on $\partial\Omega$. We shall denote by
$\|\cdot\|_\star$ the norm in
$\mathcal{L}(H^{1/2}(\Sigma),H^{-1/2}(\Sigma))$ defined by
\[
   \|T\|_\star = \sup \left\{
   \langle T\psi,\phi\rangle \ \Big|\
       \psi,\phi \in H^{1/2}_{co}(\Sigma),
    \|\psi\|_{H^{1/2}_{co}(\Sigma)}
        = \|\phi\|_{H^{1/2}_{co}(\Sigma)}=1 \right\} .
\]

\subsection{Notation and definitions}

For every $x\in \mathbb{R}^3$ we set $x=(x',x_3)$ where $x'\in
\mathbb{R}^2$ and $x_3\in\mathbb{R}$. For every $x\in \mathbb{R}^3$,
$r$ and $L$ positive real numbers we denote by $B_r(x)$, $B'_r(x')$
and $Q_{r,L}$ the open ball in $\mathbb{R}^3$ centered at $x$ of
radius $r$, the open ball in $\mathbb{R}^2$ centered at $x'$ of radius
$r$ and the cylinder $B'_r(x')\times(x_3-Lr,x_3+Lr)$, respectively;
$B_r(0)$, $B'_r(0)$ and $Q_{r,L}(0)$ will be denoted by $B_r$, $B'_r$
and $Q_{r,L}$, respectively. We will also write
$\mathbb{R}_+^3=\{(x',x_3)\in\mathbb{R}^3:x_3>0\}$,
$\mathbb{R}_-^3=\{(x',x_3)\in\mathbb{R}^3:x_3<0\}$,
$B^+_r=B_r\cap\mathbb{R}^3_+$, and $B^-_r=B_r\cap\mathbb{R}^3_-$. For
any subset $D$ of $\mathbb{R}^3$ and any $h>0$, we let
\[
   (D)_h = \{ x \in D\ |\ \operatorname{dist}(x,
            \mathbb{R}^3 \setminus D) > h \} .
\]

\begin{definition}
Let $\Omega$ be a bounded domain in $\mathbb{R}^3$. We say that a portion $\Sigma\subset\partial\Omega$ is of Lipschitz class with constants $r_0>0,L\geq 1$ if for any point $P\in\Sigma$, there exists a rigid transformation of coordinates under which $P=0$ and
\[
   \Omega \cap Q_{r_0,L} = \{ (x',x_3) \in Q_{r_0,L}\ |\
          x_3 > \psi(x') \} ,
\]
where $\psi$ is a Lipschitz continuous function in $B'_{r_0}$ such that
\[
   \psi(0) = 0\text~{and}~\|\psi\|_{C^{0,1}(B'_{r_0})}\leq L r_0 .
\]
We say that $\Omega$ is of Lipschitz class with constants $r_0$ and
$L$ if $\partial\Omega$ is of Lipschitz class with the same constants.
\end{definition}

\subsection{Main assumptions}

Let $A,L,\alpha_0,\beta_0,\gamma_0,N$ be given positive numbers such
that $N\in\mathbb{N}$, $\alpha_0\in(0,1)$, $\beta_0\in(0,2)$,
$\gamma_0\in(0,1)$ and $L>1$. We shall refer to them as the prior
data.

In the sequel we will introduce a various constants that we will
always denote by $C$.  The values of these constants might differ from
one another, but we will always have $C > 1$.

\begin{assumption}[\cite{BFV}]\label{A1}
The domain $\Omega\subset\mathbb{R}^3$ is open and bounded with
\[
   |\Omega|\leq A,
\]
and
\[
   \bar{\Omega}=\cup_{j=1}^N\bar{D}_j,
\]
where $D_j,j=1,\ldots,N$ are connected and pairwise non-overlapping
open subdomains of Lipschitz class with constants $1,L$. Moreover,
 there exists a region, say $D_1$, such that $\partial
D_1\cap\partial\Omega$ contains an open flat part, $\Sigma$, and that
for every $j\in\{2,\ldots,N\}$ there exist
$j_1,\ldots,j_M\in\{1,\ldots,N\}$ such that
\[D_{j_1}=D_1,\quad D_{j_M}=D_j\]
and, for every $k=2,\ldots,M$
\[\partial D_{j_{k-1}}\cap\partial D_{j_k}\]
contains a flat portion $\Sigma_k$ such that
\[\Sigma_k\subset\Omega,~\text{for all}~k=2,\ldots,M.\]
Furthermore, for $k=1,\ldots,M$, there exists $P_k\in\Sigma_k$ and a rigid transformation of coordinates such that $P_k=0$ and
\[\Sigma_k\cap Q_{1/3,L}=\{x\in Q_{1/3,L}:x_3=0\},\]
\[D_{j_k}\cap Q_{1/3,L}=\{x\in Q_{1/3,L}:x_3<0\},\]
\[D_{j_{k-1}}\cap Q_{1/3,L}=\{x\in Q_{1/3,L}:x_3>0\};\]
here, we set $\Sigma_1=\Sigma$. We will refer to
$D_{j_1},\ldots,D_{j_M}$ as a chain of subdomains connecting $D_1$ to
$D_j$. For any $k\in\{1,\ldots,M\}$ we will denote by $n_k$ the
exterior unit vector to $\partial D_k$ at $P_k$.
\end{assumption}

An example of such a domain partition with Lipschitz class subdomains
is an unstructured tetrahedral mesh.

\begin{figure}[h]
\begin{center}
\centering
 \includegraphics[width=2 in]{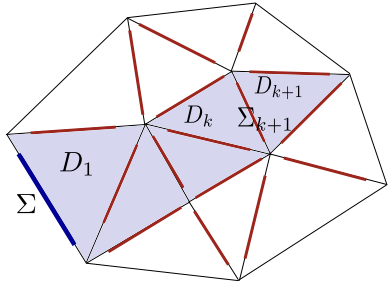}
\caption{A domain partition including $D_1$.}
\end{center}
\end{figure}

\begin{assumption}\label{A2}
The stiffness tensor, $\mathbb{C}$, is isotropic and piecewise
constant, that is,
\[
   \mathbb{C} = \sum_{j=1}^N \mathbb{C}_j\chi_{D_j}(x) ,\quad
   \mathbb{C}_j = \lambda_j I_3 \otimes I_3
               + 2 \mu_j \mathbb{I}_{sym} ,
\]
where the constants $\lambda_j$ and $\mu_j$ satisfy (cf.~(\ref{1.2}))
\begin{equation}\label{coefficient condition}
   0 < \alpha_0 \leq \mu_j \leq \alpha_0^{-1} ,\quad
   \lambda_j \leq \alpha_0^{-1} ,\quad
   2 \mu_j + 3 \lambda_j \geq \beta_0 > 0,\ j=1,\ldots,N .
\end{equation}
The density, $\rho$, is of the form,
\[
   \rho = \sum_{j=1}^N \rho_j \chi_{D_j}(x) ,
\]
where the constants $\rho_j$ satisfy (cf.~(\ref{1.3}))
\[
   0 \leq \rho_j \leq \gamma^{-1}_0 ,\ j=1,\ldots,N .
\]
\end{assumption}

\begin{assumption}\label{A3}
Let $\lambda_1^0$ be the smallest Dirichlet eigenvalue of operator
\\ $-\operatorname{div}(\mathbb{C}_0\hat{\nabla}u)$ in $\Omega$ as
before,
\[
   \omega^2 \leq \frac{\gamma_0\lambda_1^0}{2} .
\]
\end{assumption}

\subsection{Statement of the main result}

We define for any set $D \in \mathbb{R}^3$,
\[
   d_D((\mathbb{C}^1,\rho^1),(\mathbb{C}^2,\rho^2))
       = \max\{ \|\lambda^1-\lambda^2\|_{L^{\infty}(D)},
    \|\mu^1-\mu^2\|_{L^{\infty}(D)},
           \|\rho^1-\rho^2\|_{L^{\infty}(D)} \} .
\]

\begin{theorem}\label{main1}
Let $(\mathbb{C}^{1,2},\rho^{1,2})$ satisfy Assumption~\ref{A2}. Let
$\Omega$ and $\Sigma$ satisfy Assumption~\ref{A1} and $\omega$ satisfy
Assumption~\ref{A3}. If $\Lambda_{\mathbb{C}^2,\rho^2}^{\Sigma} =
\Lambda_{\mathbb{C}^1,\rho^1}^{\Sigma}$ then $\mathbb{C}^1 =
\mathbb{C}^2$ and $\rho^1 = \rho^2$. Moreover, there exists a positive
constant $C$ depending on $L$, $A$, $N$, $\alpha_0$, $\beta_0$,
$\gamma_0$ and $\lambda_1^0$ only, such that
\begin{equation}
   d_{\Omega}((\mathbb{C}^1,\rho^1),
         (\mathbb{C}^2,\rho^2)) \leq
   C \, \|\Lambda_{\mathbb{C}^1,\rho^1}^{\Sigma}
       - \Lambda_{\mathbb{C}^2,\rho^2}^{\Sigma}\|_\star .
\end{equation}
\end{theorem}

\medskip

In preparation of the proof, we introduce the forward map associated
with the inverse problem. We let $\underline{L} :=
(\lambda_1,\ldots,\lambda_N, \mu_1,\ldots,\mu_N,
\rho_1,\ldots,\rho_N)$ denote a vector in $\mathbb{R}^{3N}$ and
$\mathcal{A}$ stand for the open subset of $\mathbb{R}^{3N}$ defined
by
\begin{equation}
   \mathcal{A} := \! \left\{
   \underline{L} \in \mathbb{R}^{2N}\ \Big|\
   \frac{\alpha_0}{2}<\mu_j< \frac{2}{\alpha_0}, \lambda_j<\frac{2}{\alpha_0},2\mu_j+3\lambda_j> \frac{\beta_0}{2},\frac{\gamma_0}{2}<\rho_j< \frac{2}{\gamma_0},j=1,\ldots,N \right\}.
\end{equation}
For each vector $\underline{L}\in\mathcal{A}$ we can define a
piecewise constant stiffness tensor $\mathbb{C}_{\underline{L}}$, and
a density $\rho_{\underline{L}}$, with
\[\|\underline{L}\|_\infty=\max_{j=1,\ldots,N}\{\sup\{|\lambda_j|,\mu_j,|\rho_j|\}\}.\]
The forward map is defined as
\begin{equation}
   F :\ \mathcal{A} \rightarrow
       \mathcal{L}(H^{1/2}_{co}(\Sigma),H^{-1/2}_{co}(\Sigma)) ,
\quad
   \underline{L} \rightarrow F(\underline{L})
       = \Lambda_{\mathbb{C}_{\underline{L}},
          \rho_{\underline{L}}}^\Sigma .
\end{equation}
We can identify $F$ with a map $\tilde{F} : \mathcal{A} \rightarrow
\mathcal{B}$ upon identifying $\tilde{F}(\underline{L})$ with the
bilinear form, $\tilde{\Lambda}_{\mathbb{C}_{\underline{L}},
  \rho_{\underline{L}}}^{\Sigma}$, on $H^{1/2}_{co}(\Sigma) \times
H^{-1/2}_{co}(\Sigma)$ (cf.~(\ref{eq:bilform})); $\mathcal{B}$ is the
Banach space of this bilinear form with the standard norm.  In the
sequel, we will write $F$ and $\Lambda_{\mathbb{C}_{\underline{L}},
  \rho_{\underline{L}}}^{\Sigma}$ instead of $\tilde{F}$ and
$\tilde{\Lambda}_{\mathbb{C}_{\underline{L}},
  \rho_{\underline{L}}}^{\Sigma}$. We denote
\[
   \mathbf{K}:=\{\underline{L}\in\mathcal{A}\ |\ \alpha_0\leq\mu_j\leq\alpha_0^{-1},\lambda_j\leq\alpha_0^{-1},
2\mu_j+3\lambda_j\geq\beta_0,0\leq\rho_j\leq\gamma_0^{-1},j=1,\ldots,N\}.
\]
Then the stability estimate in Theorem \ref{main1} can be stated as
follows:
\[
   \|\underline{L}^1 - \underline{L}^2\|_\infty
       \leq C \|F(\underline{L}^1) - F(\underline{L}^2)\|_\star ,
\]
for every $\underline{L}^1,\underline{L}^2$ in $\mathbf{K}$. We note
that Theorem \ref{main1} implies that $F$ is injective and that its
inverse is Lipschitz continuous.

\begin{remark}
Assumption~\ref{A3} in Theorem \ref{main1} can be relaxed to include
any $\omega$ that is not in the set of eigenfrequencies. Then the
constant $C$ will also depend on the distance between $\omega$ and the
set of eigenfrequencies.
\end{remark}

\section{Preliminary results}

Here, we follow Beretta \textit{et al.} \cite{BFV, BFMRV}. We
summarize the relevant results in their work and adapt them to the
time-harmonic problem. We begin this section with Alessandrini's
identity \cite{Al1, Is}. We let $u_k$ be solutions to
\[
   \operatorname{div}(\mathbb{C}^k\hat{\nabla}u_k)+\rho^k\omega^2u_k=0\quad\text{in}~\Omega
\]
for $k=1,2$, where $\mathbb{C}^k, \rho^k$ satisfy Assumption \ref{A2}. Then
\begin{equation}\label{alessandrini}
\int_\Omega\left((\mathbb{C}^1-\mathbb{C}^2)\hat{\nabla}u_1:\hat{\nabla}u_2-(\rho^1-\rho^2)\omega^2u_1\cdot u_2\right)\mathrm{d}x=\langle(\Lambda_{\mathbb{C}^1,\rho^1}-\Lambda_{\mathbb{C}^2,\rho^2})u_1,u_2\rangle.
\end{equation}

\subsection{Fr\'{e}chet differentiability of $F$}
~~\\
Here, we prove the Fr\'{e}chet differentiability of the forward map, $F$.
\begin{proposition}
Under Assumptions \ref{A1}, \ref{A2} and \ref{A3}, the map
\[F:\mathcal{A}\rightarrow\mathcal{L}(H^{1/2}_{co}(\Sigma),H^{-1/2}_{co}(\Sigma))\]
 is Frech\'{e}t differentiable in $\mathcal{A}$ and
\begin{equation}\label{derivative}
\langle DF(\underline{L})[\underline{H}]\psi,\phi\rangle=\int_{\Omega}\left(\mathbb{H}\hat{\nabla}u_{\underline{L}}:\hat{\nabla}v_{\underline{L}}-h\omega^2u_{\underline{L}}\cdot v_{\underline{L}}\right)\mathrm{d}x,
\end{equation}
where $\mathbb{H}=\mathbb{C}_{\underline{H}},h=\rho_{\underline{H}}$.
Moreover, $DF :\ \mathcal{A}\rightarrow\mathcal{L}(\mathbb{R}^{3N},
\mathcal{L}(H^{1/2}_{co}(\Sigma),H^{-1/2}_{co}(\Sigma)))$ is Lipschitz
continuous with Lipschitz constant $C_{DF}$ depending on $A$, $L$,
$\alpha_0$, $\beta_0$, $\gamma_0$, $\lambda_1^0$ only.
\end{proposition}

\begin{proof}
Fix $\underline{L}\in\mathcal{A}$ and let $\underline{H}\in\mathbb{R}^{3N}$ such that $\|\underline{H}\|_\infty$ is sufficiently small. By (\ref{alessandrini}) we have
\[\langle(F(\underline{L}+\underline{H})-F(\underline{L}))\psi,\phi\rangle=\int_{\Omega}\mathbb{H}\hat{\nabla}u_{\underline{L}+\underline{H}}:\hat{\nabla}v_{\underline{L}}\mathrm{d}x-\int_{\Omega} h\omega^2u_{\underline{L}+\underline{H}}\cdot v_{\underline{L}}\mathrm{d}x.\]
Hence, by setting
\begin{multline}
\eta:=\langle(F(\underline{L}+\underline{H})-F(\underline{L}))\psi,\phi\rangle-\int_{\Omega}\mathbb{H}\hat{\nabla}u_{\underline{L}}:\hat{\nabla}v_{\underline{L}}\mathrm{d}x+\int_{\Omega} h\omega^2u_{\underline{L}}\cdot v_{\underline{L}}\mathrm{d}x
\\
 = \int_{\Omega}\mathbb{H}\hat{\nabla}(u_{\underline{L}+\underline{H}}-u_{\underline{L}}):\hat{\nabla}v_{\underline{L}}\mathrm{d}x-\int_{\Omega}h\omega^2(u_{\underline{L}+\underline{H}}-u_{\underline{L}})\cdot v_{\underline{L}}\mathrm{d}x,
\end{multline}
we find that
\begin{equation}\label{49}
|\eta|\leq C\|\underline{H}\|_\infty\|\nabla(u_{\underline{L}+\underline{H}}-u_{\underline{L}})\|_{L^2(\Omega)}\|\phi\|_{H^{1/2}_{co}(\Sigma)},
\end{equation}
where $C$ depends on $A,L,\alpha_0,\beta_0,\gamma_0,\lambda_1^0$
only. We estimate $\|\nabla(u_{\underline{L} + \underline{H}} -
u_{\underline{L}})\|_{L^2(\Omega)}$. We observe that
$w:=u_{\underline{L}+\underline{H}}-u_{\underline{L}}$ is the solution
to
\begin{equation}
\begin{cases}
\operatorname{div}(\mathbb{C}_{\underline{L}}\hat{\nabla}w)+\rho\omega^2 w=-\text{div}(\mathbb{H}\hat{\nabla}u_{\underline{L}+\underline{H}})-h\omega^2u_{\underline{L}+\underline{H}}&\text{in}~\Omega,\\
w=0&\text{on}~\partial\Omega.
\end{cases}
\end{equation}
By Proposition \ref{direct problem}, we have
\begin{equation}\label{51}
\begin{split}
\|\nabla w\|_{L^2(\Omega)}&\leq C \|w\|_{H^1(\Omega)}\\
&\leq C\|\text{div}(\mathbb{H}\hat{\nabla}u_{\underline{L}+\underline{H}})\|_{H^{-1}(\Omega)}+C\|h\omega^2u_{\underline{L}+\underline{H}}\|_{H^{-1}(\Omega)}\\
&\leq C\|\mathbb{H}\hat{\nabla}u_{\underline{L}+\underline{H}}\|_{L^2(\Omega)}+C\|h\omega^2u_{\underline{L}+\underline{H}}\|_{H^{-1}(\Omega)}\\
&\leq C\|\underline{H}\|_\infty\|u_{\underline{L}+\underline{H}}\|_{H^1(\Omega)}+C\|\underline{H}\|_\infty\|u_{\underline{L}+\underline{H}}\|_{L^{2}(\Omega)}\\
&\leq C\|\underline{H}\|_\infty\|\psi\|_{H^{1/2}_{co}(\Sigma)},
\end{split}
\end{equation}
where $C$ depends on $A,L,\alpha_0,\beta_0,\gamma_0,\lambda_1^0$.
By inserting (\ref{51}) into (\ref{49}) we get
\begin{equation}\label{52}
|\eta|\leq C\|\underline{H}\|^2_\infty\|\psi\|_{H^{1/2}_{co}(\Sigma)}\|\phi\|_{H^{1/2}_{co}(\Sigma)},
\end{equation}
that yields (\ref{derivative}).

We now prove the Lipschitz continuity of $DF$. Let $\underline{L}^1,\underline{L}^2\in\mathcal{A}$ and set
\[
\begin{split}
\xi:=&\langle(DF(\underline{L}^2)-DF(\underline{L}^1))[\underline{H}]\psi,\phi\rangle\\
=&\int_{\Omega}\left(\mathbb{H}\hat{\nabla}u_{\underline{L}^2}:v_{\underline{L}^2}-\mathbb{H}\hat{\nabla}u_{\underline{L}^1}:v_{\underline{L}^1}\right)\mathrm{d}x+\int_{\Omega}\left(h\omega^2u_{\underline{L}^2}\cdot v_{\underline{L}^2}-h\omega^2u_{\underline{L}^1}\cdot v_{\underline{L}^1}\right)\mathrm{d}x\\
=&\int_{\Omega}\mathbb{H}(\hat{\nabla}u_{\underline{L}^2}-\hat{\nabla}u_{\underline{L}^1}):\hat{\nabla}v_{\underline{L}^2}\mathrm{d}x+\int_{\Omega}\mathbb{H}\hat{\nabla}u_{\underline{L}^1}:(\hat{\nabla}v_{\underline{L}^2}-\hat{\nabla}v_{\underline{L}^1})\mathrm{d}x\\
&+\int_{\Omega}h\omega^2(u_{\underline{L}^2}-u_{\underline{L}^1})\cdot v_{\underline{L}^2}\mathrm{d}x+\int_{\Omega}h\omega^2u_{\underline{L}^1}\cdot(v_{\underline{L}^2}-v_{\underline{L}^1})\mathrm{d}x.
\end{split}\]
By reasoning as we did to derive (\ref{52}) we obtain
\[|\xi|\leq C_{DF}\|\underline{H}\|_\infty\|\underline{L}^2-\underline{L}^1\|_\infty\|\psi\|_{H^{1/2}_{co}(\Sigma)}\|\phi\|_{H^{1/2}_{co}(\Sigma)},\]
where $C_{DF}$ depends on $A,L,\alpha_0,\beta_0,\gamma_0,\lambda_1^0$.
\end{proof}

\subsection{Further notation and definitions}
~~\\
\textbf{Construction of an augmented domain and extension of $\mathbb{C}$ and $\rho$.} First we extend the domain $\Omega$ to a new domain $\Omega_0$ such that $\partial\Omega_0$ is of Lipschitz class and $B_{1/C}(P_1)\cap\Sigma\subset\Omega_0$, for some suitable constant $C\geq 1$ depending only on $L$. We proceed as in \cite{Al-Ro-R-Ve}. We set
\begin{equation}\label{rho1}
\eta_1=1/C_L, \text{where}~C_L=\frac{3\sqrt{1+L^2}}{L},
\end{equation}
and define, for every $x'\in B'_{\frac{1}{3}}$
\[
\psi^+(x')=\begin{cases}
\frac{\eta_1}{2}&\text{for}~|x'|\leq\frac{\eta_1}{4L}\\
\eta_1-2L|x'|&\text{for}~\frac{\eta_1}{4L}<|x'|\leq\frac{\eta_1}{2L}\\
0&\text{for}~|x'|>\frac{\eta_1}{2L}.
\end{cases}\]
We observe that for every $x'\in B'_{1/3}$, $|\psi^+(x')|\leq\frac{\eta_1}{2}$ and $|\nabla_{x'}\psi^+(x')|\leq 2L$. Next, we denote by
\[
   D_0 = \{ x=(x',x_3) \in Q_{1/3,L}\ |\ 0 \leq x_3<\psi^+(x')\},
\]
\[
   \Omega_0 = \Omega \cup D_0.
\]
We have
\begin{enumerate}[i)]
\item $\Omega_0$ has a Lipschitz boundary with constants $\frac{1}{3},3L$;
\item \[\Omega_0\supset Q_{1/4LC_L,L}.\]
\end{enumerate}

Let $\mathbb{C}$ be an isotropic tensor that satisfies Assumption
\ref{A2}. We extend $\mathbb{C}$ to $\Omega_0$ such that
$\mathbb{C}|_{D_0} = \mathbb{C}_0$. We
also extend $\rho$ such that $\rho|_{D_0}=1$. Then $\mathbb{C},\rho$
are of the form
\begin{equation}\label{C extension}
\mathbb{C}=\sum_{j=0}^N\mathbb{C}_j\chi_{D_j}(x),
\end{equation}
\begin{equation}\label{density extension}
\rho=\sum_{j=0}^N\rho_j\ \chi_{D_j}(x).
\end{equation}
\textbf{Construction of a walkway.} We fix $j\in\{1,\ldots,N\}$ and let $D_{j_1},\ldots,D_{j_M}$ be a chain of domains connecting $D_1$ to $D_j$. We set $D_k=D_{j_k}$, $k=1,\ldots,M$.
By \cite{Al-Ro-R-Ve} Proposition 5.5, there exists $C'_L\geq 1$ depending on $L$ only, such that $(D_k)_h$ is connected for every $k\in\{1,\ldots,M\}$ and every $h\in(0,1/C'_L)$. We introduce
\begin{equation}\label{h0}
h_0=\min\left\{\frac{1}{6},\frac{1}{C'_L},\frac{\eta_1}{8\sqrt{1+4L^2}}\right\}
\end{equation}
where $\eta_1$ is as in (\ref{rho1}).\\[0.25cm]
Furthermore
\begin{enumerate}[i)]
\item $Q_{(k)}$, $k=1,\ldots, M$, is the cylinder centered at $P_k$ such that by a rigid transformation of coordinates under which $P_k=0$ and $\Sigma_k$ belongs to the plane $\{(x',0)\}$, and $Q_{(k)}=Q_{\eta_1/4L,L}$. We also denote $Q^-_{(M)}=Q_{(M)}\cap D_{M-1}$;
\item $\mathcal{K}$ is the interior part of the set $\bigcup_{k=1}^{M-1}\bar{D}_i$;
\item $\mathcal{K}_h=\bigcup_{k=1}^{M-1}(D_i)_h$, for every $h\in(0,h_0)$;
\item
\begin{equation}\label{3.7}
\tilde{\mathcal{K}}_h=\mathcal{K}_h\cup Q_{(M)}^-\cup\bigcup_{k=1}^{M-1}Q_{(k)};
\end{equation}
\item
\[K_0=\left\{x\in D_0\ |\ \operatorname{dist}(x,\partial\Omega)>\frac{\eta_1}{8}\right\}.\]
\end{enumerate}
It is straightforward to verify that $\tilde{K}_h$ is connected and of Lipschitz class for every $h\in (0,h_0)$ and that
\begin{equation}
K_0\supset B'_{\eta_1/4L}(P_1)\times\left(\frac{\eta_1}{8},\frac{\eta_1}{4}\right).
\end{equation}

\begin{figure}[h]
\begin{center}
\centering
 \includegraphics[width=2 in]{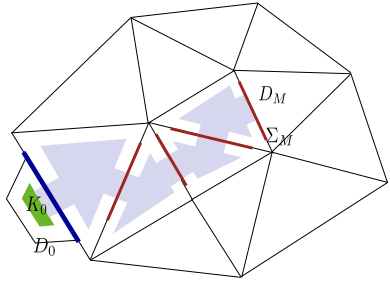}
\caption{A path or walkway.}
\end{center}
\end{figure}

\subsection{Existence of singular solutions}\label{section 3.2}
~~\\
Next, we construct singular solutions to the system describing
time-harmonic elastic waves. We prove the stability estimates for our
inverse problems by studying the behavior of singular solutions.

\subsubsection{Static fundamental solution in the biphase laminate}

In order to construct singular solutions, we make use of special
fundamental solutions constructed by Rongved \cite{R} for isotropic
biphase laminates. Consider
\[\mathbb{C}_b=\mathbb{C}^+\chi_{\mathbb{R}^3_+}+\mathbb{C}^-\chi_{\mathbb{R}^3_-},\]
where $\mathbb{C}^+$ and $\mathbb{C}^-$ are constant isotropic stiffness tensors given by
\[\mathbb{C}^+=\lambda I_3\otimes I_3+2\mu\mathbb{I}_{sym},~~\mathbb{C}^-=\lambda' I_3\otimes I_3+2\mu'\mathbb{I}_{sym},\]
with $\lambda,\mu$ and $\lambda',\mu'$ satisfying (\ref{coefficient
  condition}).

By \cite{R}, there exists a fundamental solution $\Gamma:\{(x,y)\ |\ x\in\mathbb{R}^3,y\in\mathbb{R}^3,x\neq y\}\rightarrow\mathbb{R}^{3\times 3}$ such that
\[
\operatorname{div}(\mathbb{C}_b\hat{\nabla}\Gamma(\cdot,y))=-\delta_y
I_3.
\]
Here $\delta_y$ is the Dirac distribution concentrated at $y$. We
point out some properties of $\Gamma$. First of all, it is a
fundamental solution, in the sense that $\Gamma(x,y)$ is continuous in
$\{(x,y)\in\mathbb{R}^3\times\mathbb{R}^3\ |\ x \neq y\}$,
$\Gamma(x,\cdot)$ is locally integrable in $\mathbb{R}^3$ for all
$x\in\mathbb{R}^3$, and, for every vector valued function $\phi\in
C_0^\infty(\mathbb{R}^3)$, we have
\[\int_{\mathbb{R}^3}\mathbb{C}_b\hat{\nabla}\Gamma(\cdot,y):\hat{\nabla}\phi\,\mathrm{d}x=\phi(y).\]
Furthermore, for every $x,y\in \mathbb{R}^3, x\neq y$, we have
\[|\Gamma(x,y)|\leq\frac{C}{|x-y|}
\]
and
\[|\nabla\Gamma(x,y)|\leq\frac{C}{|x-y|^2},\]
while for any $r>0$,
\begin{equation}\label{Funda solution estimate}
\|\nabla\Gamma(\cdot,y)\|_{L^2(\mathbb{R}^3\setminus B_r(y))}\leq\frac{C}{r^{1/2}},
\end{equation}
where $C$ depends on $\alpha_0,\beta_0$ only.

\subsubsection{Time-harmonic singular solutions}

Let $\mathfrak{F}$ denote the union of the flats parts of
$\cup_{j=1}^N\partial D_j$. Let $\mathcal{G}=\cup_{j=0}^N\partial
D_j\setminus\mathfrak{F}$. Let
$\mathbb{C}=\sum_{j=0}^N\mathbb{C}_j\chi_{D_j}$ where the tensors
$\mathbb{C}_j$ satisfy Assumption \ref{A2}. Let
$y\in\Omega_0\setminus\mathcal{G}$ and let
$r=\min(1/4,\text{dist}(y,\mathcal{G}\cup\partial\Omega_0))$. Then, in
the ball $B_r(y)$, either $\mathbb{C}$ is constant,
$\mathbb{C}=\mathbb{C}_j$ or
$\mathbb{C}=\mathbb{C}_j+(\mathbb{C}_{j+1}-\mathbb{C}_{j})\chi_{\{x_3>a\}}$
for some $a$ with $|a|<r$. We write
\begin{equation*}
\mathbb{C}_{y}=\begin{cases}
\mathbb{C}_j~~~\text{if}~\mathbb{C}=\mathbb{C}_j~\text{in}~B_r(y),\\
\mathbb{C}_j+(\mathbb{C}_{j+1}-\mathbb{C}_j)\chi_{\{x_3>a\}}~~ \text{otherwise},
\end{cases}
\end{equation*}
and consider the biphase fundamental solution satisfying
\begin{equation*}
\operatorname{div}(\mathbb{C}_{y}\hat{\nabla}\Gamma(\cdot,y))=-\delta_{y}I_3~\text{in}~~ \mathbb{R}^3.
\end{equation*}

\begin{proposition}\label{Singular construction}
Let $\Omega_0$, $\mathbb{C}$ and $\omega$ satisfy Assumptions
\ref{A1}, \ref{A2} and \ref{A3}. Then, for
$y\in\Omega_0\setminus\mathcal{G}$, there exists only one function
$G(\cdot,y)$, which is continuous in $\Omega\setminus\{y\}$, such that
\begin{equation}
\int_{\Omega_0}\left(\mathbb{C}\hat{\nabla}G(\cdot,y):\hat{\nabla}\phi-\rho\omega^2 G(\cdot,y)\cdot\phi\right)\mathrm{d}x=\phi(y),~~\forall\phi\in C^\infty_0(\Omega_0),
\end{equation}
and
\begin{equation*}
G(\cdot,y)=0~~\text{on}~\partial\Omega_0.
\end{equation*}
Furthermore, if
$\operatorname{dist}(y,\mathcal{G}\cup\partial\Omega_0)\geq\frac{1}{c_1}$
for some $c_1>1$ then
\begin{eqnarray}
\|G(\cdot,y)-\Gamma(\cdot,y)\|_{H^1(\Omega_0)}\leq C\label{singular estimate1},\\
\|G(\cdot,y)\|_{H^1(\Omega_0\setminus B_r(y))}\leq Cr^{-1/2}\label{singular estimate2},\\
\|G(\cdot,y)\|_{L^2(\Omega_0)}\leq C\label{singular estimate3},
\end{eqnarray}
where $C$ depends on $\alpha_0,\beta_0,A,L,\gamma_0,\lambda_1^0$ and on $c_1$.
\end{proposition}

\subsection{Unique continuation for the system describing
            time-harmonic elastic waves}

We state a quantitative estimate of unique continuation. We will omit
the proof of this estimate since it is a minor modification of the
proof of a similar estimate for the Lam\'{e} system of elasticity
\cite{BFV}.

\begin{proposition}\label{quantitative estimate}
Let $\epsilon_1,E_1$ and $h$ be positive numbers, $h<h_0$, where $h_0$ is defined in (\ref{h0}). Let $v\in H^1_{loc}(\mathcal{K})$ be a solution to
\[\operatorname{div}(\mathbb{C}\hat{\nabla}v)+\rho\omega^2 v=0~~\text{in}~~\mathcal{K},\]
such that
\[\|v\|_{L^\infty(K_0)}\leq\epsilon_1
\]
and
\begin{equation}\label{prior bound}|v(x)|\leq E_1\left(\operatorname{dist}(x,\Sigma_M)\right)^{-\gamma}~~ \text{for every}~x\in\mathcal{K}_{h/2}.\end{equation}
Then
\begin{equation}
|v(\tilde{x})|\leq Cr^{-3/2-\gamma}\epsilon_1^{\tau_r}(E_1+\epsilon_1)^{1-\tau_r},
\end{equation}
where $r\in(0,\frac{1}{C})$, $\tilde{x}=P_M+rn_M$,
\[\tau_r=\tilde{\theta}r^{\delta},\]
and $C$, $\delta$ and $\tilde{\theta}$ with $0<\tilde{\theta}<1$
depend on $A$, $L$, $\alpha_0$, $\beta_0$, $\gamma_0$ and $N$.
\end{proposition}

\noindent Therefore, if the solution to the system of time-harmonic elastic waves is small in a subdomain of $\mathcal{K}$, and has a priori bound (\ref{prior bound}), then it is also small in $\mathcal{K}$. The above proposition gives a quantitative estimates on how the smallness propagates.

\section{Proof of the main result}

In this section we prove the main result that consists of showing
the uniform continuity for $DF$ and $F^{-1}$, and establishing a lower
bound for $DF$. These results together with the Fr\'{e}chet
differentiability of $F$ establish Theorem \ref{main1} by Proposition
5 of \cite{BaV}.

\subsection{Injectivity of $F|_{\mathbf{K}}$ and uniform continuity of
            $(F|_{\mathbf{K}})^{-1}$}

Let
\begin{equation}\label{continuity1}
\sigma(t)=\begin{cases}
|\log t|^{-\frac{1}{8\delta}}~~\text{for}~~0<t<\frac{1}{e}\\
t-\frac{1}{e}+1~~\text{for}~~t\geq\frac{1}{e}
\end{cases}
\end{equation}
and
\[\sigma_1(t)=(\sigma(t))^{1/5}.\]

\begin{theorem}\label{log stability}
For every $\underline{L}^1,\underline{L}^2\in\mathbf{K}$ the following inequality holds true,
\begin{equation}\label{54}
\|\underline{L}^1-\underline{L}^2\|_\infty\leq C_*\sigma_1^N(\|F(\underline{L}^1)-F(\underline{L}^2)\|_\star)
\end{equation}
where $C_*$ is a constant depending on $A,L,\alpha_0,\beta_0,\gamma_0,\lambda_1^0,N$.
\end{theorem}


Let $j\in\{1,\ldots,N\}$ be such that
\[d_{D_j}((\mathbb{C}_{\underline{L}^1},\rho_{\underline{L}^1}),(\mathbb{C}_{\underline{L}^2},\rho_{\underline{L}^2}))=d_{\Omega_0}((\mathbb{C}_{\underline{L}^1},\rho_{\underline{L}^1}),(\mathbb{C}_{\underline{L}^2},\rho_{\underline{L}^2})),\]
and let $D_{j_1},\ldots,D_{J_M}$ be a chain of domains connecting $D_1$ to $D_j$. For the sake of simplicity of notation, set $D_k=D_{j_k}$. Let $\mathcal{W}_k=\text{Int}(\cup_{j=0}^k\overline{D}_j)$, $\mathcal{U}_k=\Omega_0\setminus\mathcal{W}_k$, for $k=1,\ldots,M-1$. The stiffness tensors $\mathbb{C}_{\underline{L}^1}$ and $\mathbb{C}_{\underline{L}^2}$ are extended as in (\ref{C extension}) to all of $\Omega_0$. The densities $\rho_{\underline{L}^1}$ and $\rho_{\underline{L}^2}$ are extended as in (\ref{density extension}). We set $\mathbb{C}:=\mathbb{C}_{\underline{L}^1}$, $\bar{\mathbb{C}}:=\mathbb{C}_{\underline{L}^2}$, $\rho:=\rho_{\underline{L}^1}$ and $\bar{\rho}:=\rho_{\underline{L}^2}$. Finally, let $\tilde{K}_k=\tilde{K}_h\cap\mathcal{W}_k$ and for $y,z\in\tilde{K}_k$ define the matrix-valued function
\[\mathcal{S}_k(y,z):=\int_{\mathcal{U}_k}\left((\mathbb{C}-\bar{\mathbb{C}})\hat{\nabla}G(x,y):\hat{\nabla}\bar{G}(x,z)-(\rho-\bar{\rho})\omega^2G(x,y)\cdot\bar{G}(x,z)\right)\mathrm{d}x,\]
the entries of which are given by
\[\begin{split}&\mathcal{S}_k^{(p,q)}(y,z)\\
:=&\int_{\mathcal{U}_k}\left((\mathbb{C}-\bar{\mathbb{C}})\hat{\nabla}G^{(p)}(x,y):\hat{\nabla}\bar{G}^{(q)}(x,z)
-(\rho-\bar{\rho})\omega^2G^{(p)}(x,y)\cdot\bar{G}^{(q)}(x,z)\right)\mathrm{d}x,\end{split}\]
$p,q=1,2,3$, where $G^{(p)}(\cdot,y)$ and $\bar{G}^{(q)}(,z)$ denote respectively the $p$-th columns and the $q$-th columns of the singular solutions corresponding to $\mathbb{C},\rho$ and $\bar{\mathbb{C}},\bar{\rho}$.
From (\ref{singular estimate2}) we have that
\[|\mathcal{S}_k^{(p,q)}(y,z)|\leq C(d(y)d(z))^{-1/2}\ \text{for all}\ y,z\in\tilde{\mathcal{K}}_k,\]
where the constant $C$ depends on the a priori parameters only and $d(y)=d(y,\mathcal{U}_k)$ and $d(z)=d(z,\mathcal{U}_k)$.

First, following a similar argument in \cite{BFV}, we have the following two propositions:

\begin{proposition}\label{prop 4.4}
For all $y,z\in\tilde{\mathcal{K}}_k$ we have that
$\mathcal{S}_k^{(\cdot,q)}(\cdot,z)$,
$\mathcal{S}_k^{(p,\cdot)}(y,\cdot)$, belong to
$H^1_{loc}(\tilde{\mathcal{K}}_k)$ and for any $q\in\{1,2,3\}$,
\begin{equation}\label{Sq system}
\operatorname{div}(\mathbb{C}\hat{\nabla}\mathcal{S}_k^{(\cdot,q)}(\cdot,z))+\rho\omega^2\mathcal{S}_k^{(\cdot,q)}(\cdot,z)=0~~\text{in}~\tilde{\mathcal{K}}_k,
\end{equation}
and for any $p\in\{1,2,3\}$,
\begin{equation}\label{Sp system}
\operatorname{div}(\bar{\mathbb{C}}\hat{\nabla}\mathcal{S}_k^{(p,\cdot)}(y,\cdot))+\bar{\rho}\omega^2\mathcal{S}_k^{(p,\cdot)}(y,\cdot)=0~~\text{in}~\tilde{\mathcal{K}}_k.
\end{equation}
\end{proposition}

\begin{proposition}\label{Prop 4.5}
If for a positive $\epsilon_0$ and for some $k\in\{1,\ldots,M-1\}$
\begin{equation}\label{4.5}
|\mathcal{S}_k(y,z)|\leq\epsilon_0~\text{for every}~(y,z)\in K_0\times K_0,
\end{equation}
then
\begin{equation}\label{58}
|\mathcal{S}_k(y_r,z_{\bar{r}})|\leq Cr^{-5/2}\bar{r}^{-2}\left(\frac{\epsilon_0}{C_1+\epsilon_0}\right)^{\tau_r\tau_{\bar{r}}},
\end{equation}
where $y_r=P_{k+1}+rn_{k+1}$, $z_{\bar{r}}=P_{k+1}+\bar{r}n_{k+1}$, $P_{k+1}\in\Sigma_{k+1}$, $r,\bar{r}\in(0,1/C)$, $\tau_r=\bar{\theta}r^\delta$, $\tau_{\bar{r}}=\bar{\theta}{\bar{r}}^\delta$ and $C,C_1,\delta,\bar{\theta}\in(0,1)$ depend on $A,L,\alpha_0,\beta_0,\gamma_0$ only.
\end{proposition}

We can also prove the following

\begin{proposition}\label{Prop 4.5'}
If $(\ref{4.5})$ holds, then
\begin{equation}\label{dy1dz1Sk}
|\partial_{y_1}\partial_{z_1}\mathcal{S}_k(y_r,z_{\bar{r}})|\leq Cr^{-9/2}\bar{r}^{-3}\left(\frac{\epsilon_0}{C_1+\epsilon_0}\right)^{\tau_r\tau_{\bar{r}}},
\end{equation}
where $y_r=P_{k+1}+rn_{k+1}$, $z_{\bar{r}}=P_{k+1}+\bar{r}n_{k+1}$, $P_{k+1}\in\Sigma_{k+1}$, $r,\bar{r}\in(0,1/C)$, $\tau_r=\bar{\theta}r^\delta$, $\tau_{\bar{r}}=\bar{\theta}{\bar{r}}^\delta$ and $C,C_1,\delta,\bar{\theta}\in(0,1)$ depend on $A,L,\alpha_0,\beta_0,\gamma_0$ only.
\end{proposition}

We note that, in the above, $\partial_{y_1}$ and $\partial_{z_1}$
denote derivatives in directions lying on the interface
$\Sigma_{k+1}$.

\begin{proof}[Proof of Proposition \ref{Prop 4.5'}]

Fix $z\in K_0$ and consider the function $v(y):=\mathcal{S}^{(\cdot,q)}(y,z)$, for fixed $q$. By Proposition \ref{prop 4.4} we know that $v$ is a solution of
\[\operatorname{div}(\mathbb{C}\hat{\nabla}v(\cdot))+\rho\omega^2v(\cdot)=0~~\text{in}~~\tilde{\mathcal{K}}_k.\]
Moreover, from Proposition \ref{Singular construction}, we get
\[|v(y)|\leq C_1d(y)^{-\frac{1}{2}},~~y\in\tilde{\mathcal{K}}_k,\]
where $C_1$ depends on $A,L,\alpha_0,\beta_0,\gamma_0,\omega,\lambda_1^0$. Then, applying Proposition \ref{quantitative estimate} for $\epsilon_1=\epsilon_0$ and $E_1=C_1$, we have
\[|v(y_r)|=|\mathcal{S}^{(\cdot,q)}_k(y_r,z)|\leq Cr^{-2}\left(\frac{\epsilon_0}{C_1+\epsilon_0}\right)^{\tau_r}\]
for all $y\in B_{r/2}(y_r)$. By the gradient estimate for an elliptic system (see for example \cite{Ni}), we obtain
\[|\partial_{y_1}v(y_r)|\leq Cr^{-3}\left(\frac{\epsilon_0}{C_1+\epsilon_0}\right)^{\tau_r}.\]

We note that $\partial_{y_1}G(x,y_r)=\partial_{y_1}\Gamma_{k+1}(x,y_r)+\partial_{y_1}w(x,y_r)$, where $\partial_{y_1}w(x,y_r)$ satisfies
\[\begin{cases}\operatorname{div}\left(\mathbb{C}\hat{\nabla}_x(\partial_{y_1}w(x,y_r))\right)+\rho\omega^2\partial_{y_1}w(x,y_r)=\operatorname{div}\left((\mathbb{C}_b^{k+1}-\mathbb{C})\hat{\nabla}_x(\partial_{y_1}\Gamma_{k+1}(x,y_r))\right)\\
\hphantom{\text{div}\left(\mathbb{C}\hat{\nabla}_x(\partial_{y_1}w(x,y_r))\right)+\rho\omega^2\partial_{y_1}w(x,y_r)==}-\rho\omega^2\partial_{y_1}\Gamma_{k+1}(x,y_r)\qquad\qquad\text{in}~\Omega_0,\\
\partial_{y_1}w(x,y_r)=-\partial_{y_1}\Gamma_{k+1}(x,y_r)\qquad\qquad\qquad\qquad\qquad\qquad\qquad\qquad\qquad\qquad\text{on}~\partial\Omega_0,
\end{cases}\]
where $\Gamma_{k+1}$ is the biphase fundamental solution for stiffness
tensor
\[\mathbb{C}_b^{k+1}=\mathbb{C}_k\chi_{\mathbb{R}^3_+}+\mathbb{C}_{k+1}\chi_{\mathbb{R}^3_-}.\]
Thus $\partial_{y_1}w(\cdot,y_r)\in H^1(\mathcal{U}_k)$ and
\begin{equation}\label{1111}
\|\partial_{y_1}w(\cdot,y_r)\|_{H^1(\mathcal{U}_k)}\leq C.
\end{equation}
Moreover,
\[\begin{split}\partial_{y_1}v(y_r)&=\partial_{y_1}\mathcal{S}_k^{(\cdot,q)}(y_r,z)\\
&=\int_{\mathcal{U}_k}\left((\mathbb{C}-\bar{\mathbb{C}})\hat{\nabla}(\partial_{y_1}G(x,y_r)):\hat{\nabla}\bar{G}(x,z)-(\rho-\bar{\rho})\omega^2(\partial_{y_1}G(x,y_r))\cdot\bar{G}(x,z)\right)\mathrm{d}x,
\end{split}\]
while
\[\bar{v}(z)=\partial_{y_1}\mathcal{S}_k^{(p,\cdot)}(y_r,z),\]
is a solution to
\[\operatorname{div}(\bar{\mathbb{C}}\hat{\nabla}v(\cdot))+\bar{\rho}\omega^2v(\cdot)=0~~\text{in}~~\tilde{\mathcal{K}}_k,\]
by the same reasoning as in Proposition \ref{prop 4.4}. By (\ref{1111}) and the estimates,
\begin{equation}\label{1113}
\|\partial_{y_1}\Gamma_{k+1}(\cdot,y)\|_{L^2(\mathbb{R}^3\setminus B_r(y))}\leq Cr^{-1/2},
\end{equation}
\begin{equation}\label{1112}
\|\nabla(\partial_{y_1}\Gamma_{k+1}(\cdot,y))\|_{L^2(\mathbb{R}^3\setminus B_r(y))}\leq Cr^{-3/2},
\end{equation}
we find that
\[|\bar{v}(z)|\leq C r^{-\frac{3}{2}}d(z)^{-\frac{1}{2}}.\]
Applying Proposition \ref{quantitative estimate} with $\epsilon_1=r^{-3}\left(\frac{\epsilon_0}{C_1+\epsilon_0}\right)^{\tau_r}$ and $E_1=Cr^{-\frac{3}{2}}$, we have
\[|\bar{v}(z)|\leq C\bar{r}^{-2}r^{-\frac{9}{2}}\left(\frac{\epsilon_0}{C_1+\epsilon_0}\right)^{\tau_r\tau_{\bar{r}}},\]
for all $z\in B_{\bar{r}/2}(z_{\bar{r}})$. Then, again, by the gradient estimate,
\[|\partial_{z_1}\bar{v}(z_{\bar{r}})|\leq C\bar{r}^{-3}r^{-\frac{9}{2}}\left(\frac{\epsilon_0}{C_1+\epsilon_0}\right)^{\tau_r\tau_{\bar{r}}}.\]

Arguing in a similar way, it also follows that
\begin{multline*}
   \partial_{z_1} \partial_{y_1} \mathcal{S}_k(y_r,z_{\bar{r}})
   = \partial_{z_1} \bar{v}(z_{\bar{r}})
\\
   = \int_{\mathcal{U}_k} \Big((\mathbb{C} - \bar{\mathbb{C}})
     \hat{\nabla}(\partial_{y_1} G(x,y_r)) :
     \hat{\nabla}(\partial_{z_1} \bar{G}(x,z_{\bar{r}}))
\\
     - (\rho - \bar{\rho}) \omega^2
     (\partial_{y_1} G(x,y_r)) \cdot
           (\partial_{z_1}\bar{G}(x,z_{\bar{r}})) \Big)
        \mathrm{d}x.
\end{multline*}
This completes the proof of (\ref{dy1dz1Sk}).
\end{proof}

\begin{proof}[Proof of Theorem \ref{log stability}]
We follow a walkway and alternate between estimates for Lam\'{e} parameters and for the density. Observe that $\|F(\underline{L}^1)-F(\underline{L}^2)\|_\star=\|\Lambda_{\mathbb{C},\rho}-\Lambda_{\bar{\mathbb{C}},\bar{\rho}}\|$. We write
\[\epsilon:=\|F(\underline{L}^1)-F(\underline{L}^2)\|_\star.\]
Then using (\ref{alessandrini}), we derive that for every $y,z\in K_0$ and for $|l|,|m|=1$,
\begin{equation}\label{59}
\left|\int_{\Omega}\left((\mathbb{C}-\bar{\mathbb{C}})(x)\hat{\nabla}G(x,y)l:\hat{\nabla} \bar{G}(x,z)m-(\rho-\bar{\rho})(x)\omega^2G(x,y)l\cdot \bar{G}(x,z)m\right)\mathrm{d}x\right|\leq C\epsilon,
\end{equation}
where $C$ depends on $\alpha_0,\beta_0,\gamma_0,\omega, A, L$.
Let
\[\delta_k:=\max_{0\leq j\leq k}\{\max\{|\lambda_j-\bar{\lambda}_j|,|\mu_j-\bar{\mu}_j|,|\rho_j-\bar{\rho}_j|\}\},\]
where $k\in\{0,1,\ldots,M\}$. We will prove that for a suitable, increasing sequence $\{\omega_k(\epsilon)\}_{0\leq k\leq M}$ satisfying $\epsilon\leq \omega_k(\epsilon)$ for every $k=0,\ldots, M$ we have
\[\delta_k\leq\omega_k(\epsilon)\Longrightarrow\delta_{k+1}\leq\omega_{k+1}(\epsilon), \text{for every}~k=0,\ldots,M-1.\]
Without loss of generality we can choose $\omega_0(\epsilon)=\epsilon$. Suppose now that for some $k=\{1,\ldots,M-1\}$ we have
\begin{equation}\label{60}
\delta_k\leq\omega_k(\epsilon).
\end{equation}
In the following, we estimate $\delta_{k+1}$ by first estimating $|\lambda_{k+1}-\bar{\lambda}_{k+1}|$, $|\mu_{k+1}-\bar{\mu}_{k+1}|$ and then $|\rho_{k+1}-\bar{\rho}_{k+1}|$.
Consider
\[\mathcal{S}_k(y,z):=\int_{\mathcal{U}_k}\left((\mathbb{C}-\bar{\mathbb{C}})(x)\hat{\nabla}G(x,y):\hat{\nabla}\bar{G}(x,z)-(\rho-\bar{\rho})(x)\omega^2G(x,y)\cdot \bar{G}(x,z)\right)\mathrm{d}x,\]
and fix $z\in K_0$. From Proposition \ref{Singular construction} and from (\ref{59}) we get that, for $y,z\in K_0$,
\[|\mathcal{S}_k(y,z)|\leq C(\epsilon+\omega_k(\epsilon)),\]
where $C$ depends on $A,L,\alpha_0,\beta_0,\gamma_0,\lambda_1^0,\omega$. By (\ref{58}) and choosing $\bar{r}=cr$ with $c\in[1/4,1/2]$, we find that there are constants $C_0,\delta\in(0,1)$ and $\theta_*$ depending on $A,L,\alpha_0,\beta_0,\gamma_0,\omega$ and $M$, such that for any $r<1/C_0$ and fixed $l,m\in\mathbb{R}^3$ with $|l|=|m|=1$,
\begin{equation}\label{61}
|\mathcal{S}_k(y_r,z_{\bar{r}})m\cdot l|\leq Cr^{-9/2}\varsigma\left(\omega_k(\epsilon),r\right),
\end{equation}
where
\[\varsigma(t,s)=\left(\frac{t}{1+t}\right)^{\theta_*s^{2\delta}}.\]
We choose $l=m=e_3$ and decompose
\begin{equation}\label{63}
\mathcal{S}_k(y_r,z_{\bar{r}})e_3\cdot e_3=I_1+I_2,
\end{equation}
where
\begin{multline}\label{64}
I_1=\int_{B_{r_1}\cap D_{k+1}}\!\!\!\!\Big((\mathbb{C}-\bar{\mathbb{C}})(x)\hat{\nabla}G(x,y_r)e_3:\hat{\nabla}\bar{G}(x,z_{\bar{r}})e_3
\\
- (\rho-\bar{\rho})(x)\omega^2G(x,y_r)e_3\cdot \bar{G}(x,z_{\bar{r}})e_3\Big)\mathrm{d}x,
\end{multline}
\begin{multline}
I_2=\int_{\mathcal{U}_{k+1}\setminus(B_{r_1}\cap D_{k+1})}\!\!\!\!\!\!\!\!\!\!\!\!\Big((\mathbb{C}-\bar{\mathbb{C}})(x)\hat{\nabla}G(x,y_r)e_3:\hat{\nabla}\bar{G}(x,z_{\bar{r}})e_3
\\
- (\rho-\bar{\rho})(x)\omega^2G(x,y_r)e_3\cdot \bar{G}(x,z_{\bar{r}})e_3\Big)\mathrm{d}x,
\end{multline}
with $r_1=\frac{1}{4LC_L}$. Then, from Proposition \ref{Singular construction}, we derive immediately that
\begin{equation}\label{65}
|I_2|\leq C.
\end{equation}
By (\ref{singular estimate3}), we have
\[\left|\int_{B_{r_1}\cap D_{k+1}}(\rho-\bar{\rho})(x)\omega^2G(x,y_r)e_3\cdot \bar{G}(x,z_{\bar{r}})e_3\mathrm{d}x\right|\leq C,\]
where $C$ depends on $A,L,\alpha_0,\beta_0,\gamma_0,\lambda_1^0$.
Using (\ref{singular estimate1}) and (\ref{singular estimate2}), we get
\begin{equation}\label{68}
\begin{split}
|I_1|\geq&\left|\int_{B_{r_1}\cap D_{k+1}}(\mathbb{C}^{k+1}_b-\bar{\mathbb{C}}^{k+1}_b)(x)\hat{\nabla}\Gamma_{k+1}(x,y_r)e_3:\hat{\nabla}\bar{\Gamma}_{k+1}(x,z_{\bar{r}})e_3\mathrm{d}x\right|-C\left(\frac{1}{\sqrt{r}}+1\right),\\
\end{split}
\end{equation}
where $\Gamma_{k+1}$ and $\bar{\Gamma}_{k+1}$ are the biphase fundamental solutions introduced in Subsection \ref{section 3.2} corresponding to the stiffness tensors $\mathbb{C}_b^{k+1}$ and $\bar{\mathbb{C}}_b^{k+1}$ given by
\[\mathbb{C}_b^{k+1}=\mathbb{C}_k\chi_{\mathbb{R}^3_+}+\mathbb{C}_{k+1}\chi_{\mathbb{R}^3_-},\]
\[\bar{\mathbb{C}}_b^{k+1}=\bar{\mathbb{C}}_k\chi_{\mathbb{R}^3_+}+\bar{\mathbb{C}}_{k+1}\chi_{\mathbb{R}^3_-},\]
up to a rigid coordinate transformation that maps the flat part of $\Sigma_{k+1}$ into $x_3=0$. Furthermore by (\ref{61}), (\ref{63}) and (\ref{65}) we obtain
\begin{equation}\label{69}
|I_1|\leq C\left(r^{-9/2}\varsigma\left(\omega_k(\epsilon),r\right)+1\right),
\end{equation}
where $C$ depends on $A,L,\alpha_0,\beta_0,\gamma_0,\lambda_1^0$. Hence, by (\ref{68}) and (\ref{69}) and by performing the change of variables $x=rx'$ in the integral, we get
\begin{equation}
\left|\int_{B^-_{r_1/r}}(\mathbb{C}^{k+1}_b-\bar{\mathbb{C}}^{k+1}_b)(x')\hat{\nabla}\Gamma_{k+1}(x',e_3)e_3:\hat{\nabla}\bar{\Gamma}_{k+1}(x',ce_3)e_3\mathrm{d}x'\right|\leq\delta_0\left(r\right),
\end{equation}
where
\[\delta_0\left(r\right)=C\left[r^{-7/2}\varsigma\left(\omega_k(\epsilon),r\right)+r^{1/2}\right].\]
We then follow the procedure of \cite{BFV} pp. 27-29, and obtain
\begin{equation}\label{Lameestimate}
|\lambda_{k+1}-\bar{\lambda}_{k+1}|\leq C\sigma(\omega_k(\epsilon)),~~~|\mu_{k+1}-\bar{\mu}_{k+1}|\leq C\sigma(\omega_k(\epsilon)).
\end{equation}

Next, we estimate $|\rho_{k+1}-\bar{\rho}_{k+1}|$. By Proposition \ref{Prop 4.5'}, there are constants $C_0,\delta\in(0,1)$ and $\theta_*$ depending on $A,L,\alpha_0,\beta_0,\gamma_0,\omega$ and, increasingly, on $M$, such that for any $r<1/C_0$ and fixed $l,m\in\mathbb{R}^3$ such that $|l|=|m|=1$,
\begin{equation}\label{61'}
|\partial_{y_1}\partial_{z_1}\mathcal{S}_k(y_r,y_r)m\cdot l|\leq Cr^{-15/2}\varsigma\left(\omega_k(\epsilon),r\right).
\end{equation}
We choose $l=m=e_3$, again, and decompose
\begin{equation}\label{63'}
\partial_{y_1}\partial_{z_1}\mathcal{S}_k(y_r,y_r)e_3\cdot e_3=J_1+J_2,
\end{equation}
where
\begin{multline}\label{64'}
J_1=\int_{B_{r_1}\cap D_{k+1}}\!\!\!\!\!\!\!\!\!\!\!\!\Big((\mathbb{C}-\bar{\mathbb{C}})(x)\hat{\nabla}(\partial_{y_1}G(x,y_r))e_3:\hat{\nabla}(\partial_{z_1}\bar{G}(x,y_r))e_3-
\\
- (\rho-\bar{\rho})(x)\omega^2(\partial_{y_1}G(x,y_r))e_3\cdot (\partial_{z_1}\bar{G}(x,y_r))e_3\Big)\mathrm{d}x,
\end{multline}
\begin{multline}J_2=\int_{\mathcal{U}_{k+1}\setminus(B_{r_1}\cap D_{k+1})}
\!\!\!\!\!\!\!\!\!\!\!\!\Big((\mathbb{C}-\bar{\mathbb{C}})(x)\hat{\nabla}(\partial_{y_1}G(x,y_r))e_3:\hat{\nabla}(\partial_{z_1}\bar{G}(x,y_r))e_3-
\\
- (\rho-\bar{\rho})(x)\omega^2(\partial_{y_1}G(x,y_r))e_3\cdot (\partial_{z_1}\bar{G}(x,y_r))e_3\Big)\mathrm{d}x.
\end{multline}
Then, with (\ref{1111}), (\ref{1113}), (\ref{1112}) we derive that
\begin{equation}\label{65'}
|J_2|\leq C.
\end{equation}
By estimates (\ref{1111}), (\ref{1113}), (\ref{1112}), and using that
$|\lambda_{k}-\bar{\lambda}_{k}| \leq C \omega_k(\epsilon)$,
$|\mu_{k}-\bar{\mu}_{k}| \leq C \omega_k(\epsilon)$,
$|\lambda_{k+1}-\bar{\lambda}_{k+1}| \leq C
\sigma(\omega_k(\epsilon))$ and $|\mu_{k+1}-\bar{\mu}_{k+1}|\leq C
\sigma(\omega_k(\epsilon))$, we get
\begin{equation}\label{68'}
\begin{split}
|J_1|\geq&\left|\int_{B_{r_1}\cap D_{k+1}}(\rho_{k+1}-\bar{\rho}_{k+1})\frac{\partial}{\partial y_1}\Gamma_{k+1}(x,y_r)e_3\cdot\frac{\partial}{\partial y_1}\Gamma_{k+1}(x,y_r)e_3\mathrm{d}x\right|\\
&-C\left(\frac{1}{\sqrt{r}}+\frac{\sigma(\omega_k(\epsilon))}{r^3}\right)
\\
\geq&|\rho_{k+1}-\bar{\rho}_{k+1}|\int_{B_{r_1}\cap D_{k+1}}\left|\frac{\partial}{\partial y_1}\Gamma_{k+1}(x,y_r)e_3\right|^2\mathrm{d}x-C\left(\frac{1}{\sqrt{r}}+\frac{\sigma(\omega_k(\epsilon))}{r^3}\right),
\end{split}
\end{equation}
where we have used that
\[\int_{B_{r_1}\cap D_{k+1}}\left|\frac{\partial}{\partial y_1}\Gamma_{k+1}(x,y_r)e_3\right|\left|\frac{\partial}{\partial y_1}\Gamma_{k+1}(x,y_r)e_3-\frac{\partial}{\partial y_1}\bar{\Gamma}_{k+1}(x,y_r)e_3\right|\mathrm{d}x\leq C\frac{\sigma(\omega_k(\epsilon))}{r}.\]
Furthermore, by (\ref{61'}),(\ref{63'}) and (\ref{65'}) we obtain
\begin{equation}\label{69'}
|J_1|\leq C\left(r^{-15/2}\varsigma\left(\omega_k(\epsilon),r\right)+1\right).
\end{equation}
By (\ref{68'}) and by performing the change of variables $x=rx'$ in the integral, we have
\begin{multline*}
r^{-1}|\rho_{k+1}-\bar{\rho}_{k+1}|\int_{B^-_{r_1/r}}\left|\frac{\partial}{\partial y_1}\Gamma_{k+1}(x',e_3)e_3\right|^2\mathrm{d}x'
\\
\leq C\left(\left(r^{-15/2}\varsigma\left(\omega_k(\epsilon),r\right)+1\right)+\frac{1}{\sqrt{r}}+\frac{\sigma(\omega_k(\epsilon))}{r^3}\right).
\end{multline*}
Since $r_1/r\geq C/4LC_L$ when $r\in(0,1/C)$, we have
\[\int_{B^-_{r_1/r}}\left|\frac{\partial}{\partial y_1}\Gamma_{k+1}(x',e_3)e_3\right|^2\mathrm{d}x'\geq \int_{B^-_{C/4LC_L}}\left|\frac{\partial}{\partial y_1}\Gamma_{k+1}(x',e_3)e_3\right|^2\mathrm{d}x'\geq C,\]
for some positive $C$. Then
\begin{equation*}
|\rho_{k+1}-\bar{\rho}_{k+1}|r^{-1}\leq C\left(\left(r^{-15/2}\varsigma\left(\omega_k(\epsilon),r\right)+1\right)+\frac{1}{\sqrt{r}}+\frac{\sigma(\omega_k(\epsilon))}{r^3}\right),
\end{equation*}
and thus
\begin{equation}
|\rho_{k+1}-\bar{\rho}_{k+1}|\leq \delta_1(r),
\end{equation}
where
\[\delta_1(r)=C\left[r^{-13/2}\varsigma\left(\omega_k(\epsilon),r\right)+\sqrt{r}+\frac{\sigma(\omega_k(\epsilon))}{r^2}\right].\]
If $\omega_k(\epsilon)<1/e$, we choose
\[r=\frac{|\sigma(\omega_k(\epsilon))|^{2/5}}{C},\]
and then
\begin{equation}\label{7}
|\rho_{k+1}-\bar{\rho}_{k+1}|\leq C|\sigma(\omega_k(\epsilon))|^{1/5}.
\end{equation}
Otherwise, if $\omega_k(\epsilon)\geq 1/e$, since $|\rho_{k+1}-\bar{\rho}_{k+1}|$ is bounded, we get (\ref{7}) trivially. By (\ref{Lameestimate}) and (\ref{7}), we follow the weakest estimate to get
\[\delta_{k+1}\leq\omega_{k+1}(\epsilon):=C\sigma_1(\omega_k(\epsilon)).\]
Following the way of alternatingly estimating $|\lambda-\bar{\lambda}|$, $|\mu-\bar{\mu}|$ and $|\rho-\bar{\rho}|$ along the walkay $D_1,D_2,\ldots,D_M$, and recalling that $\omega_0(\epsilon)=\epsilon$, we get (\ref{54}).

\end{proof}
The uniqueness statement in Theorem \ref{main1} is an immediate
consequence of the proposition above.

\subsection{Injectivity of $DF(\underline{L})$ and estimate from below of $DF|_{\mathbf{K}}$}

\begin{proposition}\label{inject DF}
Let
\[q_0:=\min\{\|DF(\underline{L})[\underline{H}]\|_\star\ |\ \underline{L}\in\mathbf{K},\underline{H}\in \mathbb{R}^{3N},\|\underline{H}\|_\infty=1\};\]
we have
\begin{equation}\label{54f}
(\sigma_1^N)^{-1}(1/C_\star)\leq q_0,
\end{equation}
where $C_\star>1$ depends on $A,L,\alpha_0,\beta_0,\gamma_0,\lambda_1^0$ and $N$ only.
\end{proposition}

\begin{proof}
By the definition of $q_0$ there exists an $\underline{L}_0\in\mathbf{K}$ and \[\underline{H}_0=(h_{0,1},\ldots,h_{0,N},k_{0,1},\ldots,k_{0,N},l_{0,1},\ldots,l_{0,N}),~~\|\underline{H}_0\|_\infty=1,\] such that
\begin{equation}\label{77f}
\|DF(\underline{L}_0)[\underline{H}_0]\|_\star=q_0.
\end{equation}
Therefore, by (\ref{derivative}), (\ref{77f}), we have
\begin{equation}\label{59f}
\left|\int_{\Omega}\mathbb{H}(x)\left(\hat{\nabla}G(x,y)l:\hat{\nabla} G(x,z)m-h(x)\omega^2G(x,y)l\cdot G(x,z)m\right)\mathrm{d}x\right|\leq Cq_0
\end{equation}
for every $y,z\in\mathcal{K}_0$, where $C$ depends on $\alpha_0,\beta_0,\gamma_0,\omega, A, L$, $\mathbb{H}=\mathbb{C}_{\underline{H}_0}$, $h=\rho_{\underline{H}_0}$ and $G(\cdot,y)$ denotes the singular solution corresponding to $\mathbb{C}_{\underline{L}},\rho_{\underline{L}}$. From now on the vector
\[(0, h_{0,1},\ldots,h_{0,N},0,k_{0,1},\ldots,k_{0,N},0,l_{0,1},\ldots,l_{0,N}),\]
 will still be denoted by $\underline{H}_0$.

We fix $j\in\{1,\ldots,N\}$ and let $D_{j_1},\ldots,D_{j_M}$ be a chain of domains connecting $D_1$ to $D_j$, where
\[\max\{|h_{0,j}|,|k_{0,j}|,|l_{0,j}|\}=\|\underline{H}_0\|_{\infty}=1.\]
Now, let
\[\eta_i:=\max_{0\leq j\leq i}\{\max\{|h_{0,j}|,|k_{0,j}|,|l_{0,j}|\}\},\]
where $i\in\{0,1,\ldots,M\}$. We will prove that for a suitable increasing sequence $\{\omega_i(q_0)\}_{0\leq i\leq M}$ satisfying $\epsilon\leq \omega_i(q_0)$ for every $k=0,\ldots, M$, we have
\[\delta_k\leq\omega_i(q_0)\Longrightarrow\delta_{i+1}\leq\omega_{k+1}(q_0)\ \text{for every}~i=0,\ldots,M-1.\]
Without loss of generality we can choose $\omega_0(q_0)=q_0$. Suppose now that for some $i=\{1,\ldots,M-1\}$ we obtain (\ref{77f}). Let $\mathcal{Y}_i(y,z)=\{\mathcal{Y}_i^{(p,q)}(y,z)\}_{1\leq p,q\leq 3}$ be the matrix valued function the elements of which are given by
\[\mathcal{Y}_i^{(p,q)}(y,z):=\int_{\mathcal{U}_i}\left(\mathbb{H}(x)\hat{\nabla}G^{(p)}(x,y):\hat{\nabla}G^{(q)}(x,z)-h(x)\omega^2G^{(p)}(x,y)\cdot G^{(q)}(x,z)\right)\mathrm{d}x,\]
with $z \in K_0$ fixed. From Proposition \ref{Singular construction} and from (\ref{59}) we get that, for $y,z\in K_0$,
\[|\mathcal{Y}_i(y,z)|\leq C(q_0+\omega_i(q_0)),\]
where $C$ depends on $A,L,\alpha_0,\beta_0,\gamma_0,\lambda_1^0$. Choosing $\bar{r}=cr$ with $c\in[1/4,1/2]$, as in Proposition \ref{Prop 4.5}, we have that there exists a constant $C_2$ such that for every $r\in(0,1/C_2)$,
\begin{equation}\label{61f}
|\mathcal{Y}_i(y_r,z_{\bar{r}})|\leq Cr^{-9/2}\varsigma\left(\omega_i(q_0,r)\right),
\end{equation}
where
\[\varsigma(t,s)=\left(\frac{t}{1+t}\right)^{\theta_*s^{2\delta}}.\]
We choose $l=m=e_3$, again, and decompose
\begin{equation}\label{63f}
\mathcal{Y}_k(y_r,z_{\bar{r}})e_3\cdot e_3=I_1+I_2,
\end{equation}
where
\begin{equation}\label{64f}
I_1=\int_{B_{r_1}\cap D_{i+1}}\left(\mathbb{H}(x)\hat{\nabla}G(x,y_r)e_3:\hat{\nabla}G(x,z_{\bar{r}})e_3-h(x)\omega^2\bar{G}(x,y_r)e_3\cdot G(x,z_{\bar{r}})e_3\right)\mathrm{d}x,
\end{equation}
\begin{equation}
\begin{split}I_2=\int_{\mathcal{U}_{i+1}\setminus(B_{r_1}\cap D_{i+1})}&\Big(\mathbb{H}(x)\hat{\nabla}G(x,y_r)e_3:\hat{\nabla}G(x,z_{\bar{r}})e_3\\
&-h(x)\omega^2G(x,y_r)e_3\cdot G(x,z_{\bar{r}})e_3\Big)\mathrm{d}x,\end{split}
\end{equation}
and $r_1=\frac{1}{4LC_L}$. Then, from Proposition \ref{Singular construction}, we derive that
\begin{equation}\label{65f}
|I_2|\leq C.
\end{equation}
Using (\ref{singular estimate3}), we find that
\[\left|\int_{B_{r_1}\cap D_{k+1}}h(x)\omega^2G(x,y_r)e_3\cdot G(x,z_{\bar{r}})e_3\mathrm{d}x\right|\leq C,\]
where $C$ depends on $A,L,\alpha_0,\beta_0,\gamma_0,\lambda_1^0$.
Then, by (\ref{singular estimate1}) and (\ref{singular estimate2}) we get
\begin{equation}\label{68f}
|I_1|\geq\left|\int_{B_{r_1}\cap D_{i+1}}\mathbb{H}(x)\hat{\nabla}\Gamma_{i+1}(x,y_r)e_3:\hat{\nabla}\Gamma_{i+1}(x,z_{\bar{r}})e_3\mathrm{d}x\right|-C\left(\frac{1}{\sqrt{r}}+1\right).
\end{equation}
With
(\ref{61f}), (\ref{63f}) and (\ref{65f}) we obtain
\begin{equation}\label{69f}
|I_1|\leq C\left(r^{-9/2}\varsigma\left(\omega_i(q_0),r\right)+1\right),
\end{equation}
where $C$ depends on
$A,L,\alpha_0,\beta_0,\gamma_0,\lambda_1^0$. Following the procedure 
of \cite{BFV} pp. 31-33,
we get
\begin{equation}\label{lmestimate}
|h_{0,i+1}|\leq C\sigma(\omega_i(q_0)),~~~|k_{0,i+1}|\leq C\sigma(\omega_i(q_0)).
\end{equation}

Similar to Proposition \ref{Prop 4.5'}, we find that there are constants $C_2,\delta\in(0,1)$ and $\theta_*$ depending on $A,L,\alpha_0,\beta_0,\gamma_0,\omega$ and, increasingly, on $M$, such that for any $r<1/C_2$ 
\begin{equation}\label{61'f}
|\partial_{y_1}\partial_{z_1}\mathcal{Y}_i(y_r,y_r)e_3\cdot e_3|\leq Cr^{-15/2}\varsigma\left(\omega_i(q_0,r)\right).
\end{equation}
We decompose
\begin{equation}\label{63'f}
\partial_{y_1}\partial_{z_1}\mathcal{Y}_i(y_r,y_r)e_3\cdot e_3=J_1+J_2,
\end{equation}
where
\begin{multline}\label{64'f}
J_1=\int_{B_{r_1}\cap D_{i+1}} \Big(\mathbb{H}(x)\hat{\nabla}(\partial_{y_1}G(x,y_r))e_3:\hat{\nabla}(\partial_{z_1}G(x,y_r))e_3\\
- h(x)\omega^2(\partial_{y_1}G(x,y_r))e_3\cdot (\partial_{z_1}G(x,y_r))e_3\Big)\mathrm{d}x,
\end{multline}
\begin{multline}
J_2=\int_{\mathcal{U}_{i+1}\setminus(B_{r_1}\cap D_{i+1})} \Big(\mathbb{H}(x)\hat{\nabla}(\partial_{y_1}G(x,y_r))e_3:\hat{\nabla}(\partial_{z_1}G(x,y_r))e_3\\
- h(x)\omega^2(\partial_{y_1}G(x,y_r))e_3\cdot (\partial_{z_1}G(x,y_r))e_3\Big)\mathrm{d}x.
\end{multline}
Using (\ref{1111}), (\ref{1113}), (\ref{1112}) and (\ref{lmestimate}), we get
\begin{equation}\label{65'f}
|J_2|\leq C
\end{equation}
and 
\begin{equation}\label{68'f}
\begin{split}
|J_1|\geq&\left|\int_{B_{r_1}\cap D_{i+1}}l_{0,i+1}\frac{\partial}{\partial y_1}\Gamma_{i+1}(x,y_r)e_3\cdot\frac{\partial}{\partial y_1}\Gamma_{i+1}(x,y_r)e_3\mathrm{d}x\right|-C\left(\frac{1}{\sqrt{r}}+\frac{\sigma(\omega_i(\epsilon))}{r^3}\right)\\
=&|l_{0,i+1}|\int_{B_{r_1}\cap D_{i+1}}\left|\frac{\partial}{\partial y_1}\Gamma_{i+1}(x,y_r)e_3\right|^2\mathrm{d}x-C\left(\frac{1}{\sqrt{r}}+\frac{\sigma(\omega_i(q_0))}{r^3}\right).
\end{split}
\end{equation}
Furthermore by (\ref{61'f}), (\ref{63'f}) and (\ref{65'f}), we obtain
\begin{equation}\label{69'f}
|J_1|\leq C\left(r^{-15/2}\varsigma\left(\omega_i(q_0)),r\right)+1\right).
\end{equation}
Hence, by (\ref{68'f}) and upon performing the change of variables $x=rx'$ in the integral, we obtain
\begin{multline*}
r^{-1}|l_{0,i+1}|\int_{B^-_{r_1/r}}\left|\frac{\partial}{\partial y_1}\Gamma_{i+1}(x',e_3)e_3\right|^2\mathrm{d}x'
\\
\leq C\left(\left(r^{-15/2}\varsigma\left(\omega_i(q_0)),r\right)+1\right)+\frac{1}{\sqrt{r}}+\frac{\sigma(\omega_i(q_0))}{r^3}\right).
\end{multline*}
Since $r_1/r\geq C/4LC_L$ when $r\in(0,1/C)$, we have
\[\int_{B^-_{r_1/r}}\left|\frac{\partial}{\partial y_1}\Gamma_{i+1}(x',e_3)e_3\right|^2\mathrm{d}x'\geq \int_{B^-_{C/4LC_L}}\left|\frac{\partial}{\partial y_1}\Gamma_{i+1}(x',e_3)e_3\right|^2\mathrm{d}x'\geq C.\]
Then 
\begin{equation*}
|l_{0,i+1}|r^{-1}\leq C\left(\left(r^{-15/2}\varsigma\left(\omega_i(q_0)),r\right)+1\right)+\frac{1}{\sqrt{r}}+\frac{\sigma(\omega_i(q_0))}{r^3}\right),
\end{equation*}
and thus
\begin{equation}
|l_{0,i+1}|\leq \delta_1(r),
\end{equation}
where
\[\delta_1(r)=C\left[r^{-13/2}\varsigma\left(\omega_i(q_0),r\right)+\sqrt{r}+\frac{\sigma(\omega_i(q_0))}{r^2}\right].\]
If $\omega_i(q_0)<1/e$, we choose
\[r=\frac{|\sigma(\omega_i(q_0))|^{2/5}}{C}
\]
so that
\begin{equation}\label{7f}
|l_{0,i+1}|\leq C|\sigma(\omega_i(q_0))|^{1/5}.
\end{equation}
Otherwise, if $\omega_i(q_0)\geq 1/e$, because $|l_{0,i+1}|$ is bounded, we get (\ref{7f}) trivially. Then, by (\ref{lmestimate}) and (\ref{7f}) we get
\[\eta_{i+1}\leq\omega_{i+1}(q_0):=C\sigma_1(\omega_i(q_0)).\]
Finally, by alternating the estimates for $|\lambda-\bar{\lambda}|,|\mu-\bar{\mu}|$ and $|\rho-\bar{\rho}|$, we get
\[1=\eta_M\leq C\sigma_1^M(q_0)\leq C\sigma_1^N(q_0),\]
and the statement follows.
\end{proof}

\section{Remarks on two reduced problems}

The stability estimates for the following two complementary inverse problems are immediate implications of Theorem \ref{main1}.
\begin{enumerate}[(i)]
\item \textbf{Inverse Problem S1:} For known $\rho$: determine $\mathbb{C}$ from $\Lambda_{\mathbb{C},\rho}$;
\item \textbf{Inverse Problem S2:} For known $\mathbb{C}$: determine $\rho$ from $\Lambda_{\mathbb{C},\rho}$,
\end{enumerate}
However, here, that we get much improved estimates in Theorem \ref{log
  stability}, and Proposition \ref{inject DF}. This enables us to get
better Lipschitz constants in the final Lipschitz stability estimates.

\begin{corollary}\label{log stability S1}
For every $\underline{L}^1,\underline{L}^2\in\mathbf{K}$ the following inequality holds true
\begin{equation}
\|\underline{L}^1-\underline{L}^2\|_\infty\leq C_*\sigma^N(\|F(\underline{L}^1)-F(\underline{L}^2)\|_\star)
\end{equation}
if either
\[\rho^1_i=\rho^2_i,~i=1,\cdots,N~~~~~~(\text{Problem S1})\]
or
\[\lambda^1_i=\lambda^2_i,~\mu^1_i=\mu^2_i,~i=1,\cdots,N~~~~~~(\text{Problem S2})\]
where $C_*$ is a constant depending on $A,L,\alpha_0,\beta_0,\gamma_0,\lambda_1^0,N$.
\end{corollary}

\begin{corollary}
Let
\[q_0:=\min\{\|DF(\underline{L})[\underline{H}]\|_\star|\underline{L}\in\mathbf{K},\underline{H}\in \mathbb{R}^{3N},\|\underline{H}\|_\infty=1\}.\]
We have
\begin{equation}
q_0\geq (\sigma^N)^{-1}(1/C_\star)
\end{equation}
if either
\[l_i=0,~i=1,\cdots,N~~~~~~(\text{Problem S1})\]
or
\[h_i=k_i=0,~i=1,\cdots,N~~~~~~(\text{Problem S2})\]
where $C_\star>1$, depends on $A,L,\alpha_0,\beta_0,\gamma_0,\lambda_1^0$ and $N$ only.
\end{corollary}

We note that, here, $\sigma$ replaces $\sigma_1$ in the corollaries
above. This is due to the fact that we are not dealing with the
multi-parameter identification. That is, we do not need to
alternatingly estimate coefficients of different order terms.

\bibliographystyle{siam}

\end{document}